\documentclass[a4paper, reqno]{amsart}

\usepackage[all]{xy}
\usepackage[english]{babel}
\usepackage[utf8]{inputenc}
\usepackage{amssymb,amsmath,amsthm}
\usepackage{amsfonts}
\usepackage{graphicx}
\usepackage{psfrag}
\usepackage[dvipsnames]{xcolor}
\usepackage[hidelinks]{hyperref}
\usepackage{csquotes}
\usepackage[alphabetic]{amsrefs}
\usepackage{enumitem}
\usepackage{cancel}
\usepackage[normalem]{ulem}

\allowdisplaybreaks

\newcommand{\Id}{\mathrm{Id}}
\newcommand{\Ric}{\mathrm{Ric}}

\newcommand{\e}{\epsilon}

\newcommand{\n}{\nabla}
\renewcommand{\Im}{\mathrm{Im} \,}
\renewcommand{\L}{\mathcal{L}}

\renewcommand{\o}{\omega}
\renewcommand{\O}{\mathcal O}

\renewcommand{\a}{\alpha}
\newcommand{\de}{\delta}
\renewcommand{\b}{\beta}
\renewcommand{\d}{\partial}

\newcommand{\abs}[1]{\left\lvert#1\right\rvert}

\renewcommand{\H}{\mathcal H}

\newcommand{\U}{\mathcal U}

\newcommand{\W}{\mathcal W}
\newcommand{\I}{\mathrm I}
\newcommand{\J}{\mathrm J}
\newcommand{\V}{\mathcal V}
\newcommand{\T}{\mathcal T}
\newcommand{\md}{\mathrm d}
\newcommand{\p}{\mathrm p}

\newcommand{\s}{\sigma}

\theoremstyle{plain}
\newtheorem{thm}{Theorem}[section]
\newtheorem{prop}[thm]{Proposition}
\newtheorem{lemma}[thm]{Lemma}

\theoremstyle{definition}
\newtheorem{definition}[thm]{Definition}

\newtheorem{remark}[thm]{Remark}
\newtheorem{example}[thm]{Example}

\newcommand{\R}[0]{\mathbb{R}}							% Real numbers
\newcommand{\C}[0]{\mathbb{C}}						% Complex numbers
\newcommand{\N}[0]{\mathbb{N}}						% Natural numbers
\newcommand{\Z}[0]{\mathbb{Z}}						% Integers
\makeatletter
\@namedef{subjclassname@1991}{Subject}
\@namedef{subjclassname@2000}{Subject}
\@namedef{subjclassname@2010}{2020 Mathematics Subject Classification}
\makeatother

\title[Analyticity of quasinormal modes]{Analyticity of quasinormal modes in the Kerr and Kerr-de Sitter spacetimes}

\author{Oliver Petersen}
\address{Department of Mathematics, KTH Royal Institute of Technology, Lindstedtsvägen 25, 11428 Stockholm, Sweden}
\email{oliverlp@kth.se}

\author{Andr\'{a}s Vasy}
\address{Department of Mathematics, Stanford University, Stanford, CA 94305-2125, USA}
\email{andras@math.stanford.edu}


\thanks{The first author is grateful for the support of the Deutsche Forschungsgemeinschaft (DFG, German Research Foundation) – GZ: LI 3638/1-1, AOBJ: 660735 and the Stanford Mathematics Research Center. The second author is grateful for the support of the National Science Foundation under grant number DMS-1953987.}

\begin{document}
\hbadness=100000
\vbadness=100000

\begin{abstract}
We prove that quasinormal modes (or resonant states) for linear wave equations in the subextremal Kerr and Kerr-de Sitter spacetimes are real analytic. 
The main novelty of this paper is the observation that the bicharacteristic flow associated to the linear wave equations for quasinormal modes with respect to a suitable Killing vector field has a stable radial point source/sink structure rather than merely a generalized normal source/sink structure.
The analyticity then follows by a recent result in the microlocal analysis of radial points by Galkowski and Zworski. 
The results can then be recast with respect to the standard Killing vector field.
\end{abstract}

\maketitle

\tableofcontents
\begin{sloppypar}

\section{Introduction}

When studying linear and nonlinear wave equations on black hole spacetimes, such as the Kerr spacetime and Kerr-de Sitter spacetime, quasinormal modes play a prominent role. 
Indeed, for linear equations, within certain limitations corresponding to trapped null-geodesics, solutions have an asymptotic expansion at timelike infinity in quasinormal modes. 
Such expansions, or the corresponding decay or lack thereof statements, have a long history which in the mathematics literature
goes back to S\'a Barreto and Zworski \cite{SBZ1997}, Bony and H\"afner \cite{BH2008}, Dyatlov \cites{D2011, D2012}, Vasy \cite{V2013}, Shlapentokh-Rothman \cite{SR2015}, Hintz and Vasy \cite{HV2015} and Gajic and Warnick \cite{GW2020}. 
In the physics literature the importance of these has been clear even longer, going back to Regge and Wheeler \cite{RW1957}, Vishveshwara \cite{V1970}, Zerilli \cite{Z1970}, Whiting \cite{W1989}, Kodama, Ishibashi and Seto \cite{KIS2000} and others. 
For nonlinear equations the non-decaying quasinormal modes become an obstacle to solvability; for equations with gauge freedom, such as Einstein's equation, it is non-decaying modes that are not `pure gauge' that play an analogous role \cite{HV2018}.

Quasinormal modes are solutions of the homogeneous wave equation which are eigenfunctions of the covariant derivative along appropriate Killing vector fields. 
A key consideration for applications is that for similar covariant eigenfunctions as the forcing (right hand side of the wave equation), there should be a satisfactory Fredholm theory. 
In this case the covariant eigenvalues (resonances) form a discrete set, and the corresponding eigenspaces are finite dimensional. 
Since Fredholm theory is global, this necessitates working relative to Killing vector fields with suitable global behavior.

Recently Galkowski and Zworski \cite{GZ2020} showed that quasinormal modes for non-rotating black holes are real analytic at the horizon; indeed they obtained a substantially stronger microlocal result.
In this paper we generalize their result to the case of rotating black holes whose importance is underlined by their ubiquity. 
Our proof relies crucially on the ability to {\em locally} transform the rotating black hole quasimode problem to the non-rotating one, and thus being able to apply the result of \cite{GZ2020}. 
This transformation is facilitated by locally considering analogues of quasinormal modes with respect to a different Killing vector field that is lightlike on the horizon; this change is very simple in the Kerr and Kerr-de Sitter case as we discuss below, but in fact works in general for non-degenerate Killing horizons under an additional condition, as is also described below. 
While these modes are with respect to a different Killing vector field, we can in fact relate these to the quasinormal modes with respect to the original globally well-behaved vector field to obtain the real analyticity result. 
Indeed, a key feature of the Kerr-de Sitter setting is the presence of two horizons, and the well-behaved Killing vector fields with respect to each of these horizons, while globally well-defined, are ill-behaved at the other horizon. 
Thus, it is of central importance for our approach to be able to work locally near a horizon to obtain the analyticity conclusions.

It is conceivable that the results of \cite{GZ2020} could be proven in the more general setting of \cite{V2013}, which would imply analyticity in the case of rotating black holes.
However, such a proof would be significantly more technically involved than the proof we give here.
The proof of analytic hypoellipticity in \cite{GZ2020} relies on a microlocal normal form of Haber \cite{H2014}, which in turn relies on the relevant Lagrangian (in our case the conormal bundle of the horizon) being radial with respect to the Hamilton vector field.
For rotating black holes, there is more intricate internal dynamics and one could therefore not directly apply the results of Haber.
Indeed, our reduction to \cite{GZ2020} can be considered as a way of `straightening’ the dynamics and thus bringing it to a model form.

Furthermore, it could perhaps be possible to prove analytic hypoellipticity of Keldysh-type operators more explicitly using ODE theory and separation of variables and thereby avoid referring to \cite{GZ2020}.
This was the approach by Lebeau and Zworski in \cite{LZ2019} (see also the work by Zuily in \cite{Zu2017}), where an explicit class of Keldysh type operators where studied, and certain values for the subprincipal symbol had to be excluded.
However, the main purpose of this paper is to reduce the hypoanalyticity of quasinormal modes in Kerr(-de Sitter) spacetimes to the irrotational, i.e.\ Keldysh-type, case.
Applying the ODE approach directly in the rotational case seems cumbersome and it is not clear to us whether it would work.

In the rest of the introduction we describe the precise results in the rotating black hole setting, as well as the generalization to non-degenerate Killing horizons. Then in Section~\ref{sec:coordinates} we discuss geometric aspects of these Killing horizons. 
In Section~\ref{sec:general} we then prove our general local result.
In Section~\ref{sec: proof of first main result} we use these local results to obtain a global result for joint modes of two Killing vector fields on Kerr and Kerr-de Sitter spacetimes. 
Finally, in Section~\ref{sec:improved} we show how these results imply the real analyticity of the quasinormal modes on Kerr and Kerr-de Sitter spacetimes, with modes taken with respect to the standard Killing vector field.

\subsection{Kerr and Kerr-de Sitter spacetime} \label{subsec: Kerr-de Sitter}

Fix three parameters $a \in \R$ and $m, \Lambda \geq 0$, such that the polynomial
\begin{equation} \label{eq: mu}
	\mu(r) := \left(r^2 + a^2\right)\left(1 -\frac{\Lambda r^2}3\right) - 2mr
\end{equation}
has four distinct real roots $r_- < r_C < r_e < r_c$ if $\Lambda > 0$ and two distinct real roots $r_C < r_e$ if $\Lambda = 0$.
The latter condition is equivalent to $\abs a < m$.

Assuming $\Lambda > 0$, the domain of outer communication in the sub-extermal \emph{Kerr-de Sitter spacetime} is given in Boyer-Lindquist coordinates by the real analytic spacetime
\[
	\R_t \times (r_e, r_c)_r \times S^1_\phi \times (0, \pi)_\theta,
\]
with real analytic metric
\begin{equation} \label{eq: g}
\begin{split}
	g
		&= (r^2 + a^2 \cos^2(\theta))\left( \frac{\md r^2}{\mu(r)} + \frac{\md \theta^2}{c(\theta)}\right) \\*
		&\quad + \frac{c(\theta)\sin^2(\theta)}{b^2\left(r^2 + a^2 \cos^2(\theta)\right)}\left(a \md t - \left(r^2 + a^2\right)\md \phi\right)^2 \\*
		&\quad - \frac{\mu(r)}{b^2\left(r^2 + a^2 \cos^2(\theta)\right)}\left(\md t - a \sin^2(\theta)\md \phi\right)^2,
\end{split}
\end{equation}
where
\[
	b := 1 + \frac{\Lambda a^2}3, \quad c(\theta) := 1 + \frac{\Lambda a^2}3 \cos^2(\theta).
\]
The domain of outer communication in the subextremal \emph{Kerr spacetime} is defined analogously, by passing to the limit $\Lambda = 0$.
We set $r_c = \infty$ if $\Lambda = 0$.
The Boyer-Lindquist coordinates could be thought of as spherical coordinates around the black hole, where the black hole is centered at $r = 0$.
Even though they are not defined at the north and south poles $\theta = 0$ and $\pi$, it is straightforward to check that the metric \eqref{eq: g} extends real analytically to 
\[
	M := \R_t \times (r_e, r_c)_r \times S^2_{\phi, \theta}.
\]
Since the metric extends real analytically, so does linear wave equations with a principal symbol given by the metric in these coordinates.
We refer to \cite{HV2018}*{Sec.~3}, for a more thorough discussion of the geometry of Kerr-de Sitter spacetimes.

These coordinates are singular at the roots of $\mu(r)$.
In order to define quasinormal modes, we need to extend this metric real analytically over the future event horizon and the future cosmological horizon, corresponding to the roots $r = r_e$ and $r = r_c$, respectively.
This can be done, for instance, by the following coordinate change:
\begin{align*}
	t_*
		&:= t - \Phi(r), \\
	\phi_*
		&:= \phi - \Psi(r),
\end{align*}
where $\Phi$ and $\Psi$ satisfy
\begin{align*}
		\Phi'(r)
			&= b \frac{r^2 + a^2}{\mu(r)} f(r), \\
		\Psi'(r)
			&= b \frac a{\mu(r)} f(r).
\end{align*}
In the case $\Lambda > 0$, we let $f: (r_e-\de, r_c+\de) \to \R$, for a small $\de > 0$, be a real analytic function such that
\[
	f(r_e) = - 1
\]
and
\[
	f(r_c) = 1.
\]
In the case $\Lambda = 0$, there is no cosmological horizon, so we instead assume that
\[
	\lim_{r \to \infty}f(r) = 1.
\]
The metric \eqref{eq: g} extends real analytically to the manifold 
\[
	M_* 
		:= \R_{t_*} \times (r_e-\de, r_c+\de)_r \times S^2_{\phi_*, \theta},
\]
and is given by
\begin{equation}
\begin{split}
	g_*
		&= (r^2 + a^2 \cos^2(\theta)) \frac{1 - f(r)^2}{\mu(r)}\md r^2 \\*
		&\qquad - \frac 2 b f(r)(\md t_* - a \sin^2(\theta) \md \phi_*)\md r \\*
		&\qquad - \frac{\mu(r)}{b^2\left(r^2 + a^2 \cos^2(\theta)\right)}\left(\md t_* - a \sin^2(\theta) \md \phi_* \right)^2 \\*
		&\qquad + \frac{c(\theta)\sin^2(\theta)}{b^2\left(r^2 + a^2 \cos^2(\theta)\right)}\left(a\md t_* - \left(r^2 + a^2\right) \md \phi_*\right)^2 \\*
		&\qquad + (r^2 + a^2 \cos^2(\theta))\frac{\md \theta^2}{c(\theta)}.
\end{split}
\label{eq: metric at horizons}
\end{equation}
We will throughout the paper assume that $f$ is chosen as in \cite{PV2021}*{Rmk.~1.1}, so that the hypersurfaces
\[
	\{t_* = c \} \times (r_e - \de, r_c + \de)_r \times S^2_{\phi_*, \theta}
\]
are \emph{spacelike}, for all $c \in \R$, and that $\de > 0$ is small enough so that the hypersurfaces
\begin{align*}
	&\R_{t_*} \times \{r = r_e - \de \} \times S^2_{\phi_*, \theta}, \\
	&\R_{t_*} \times \{r = r_c + \de \} \times S^2_{\phi_*, \theta}
\end{align*}
are spacelike.
The two real analytic lightlike hypersurfaces
\begin{align*}
	\H_e^+ 
		&:= \R_{t_*} \times \{r_e\} \times S^2_{\phi_*, \theta}, \\
	\H_c^+
		&:= \R_{t_*} \times \{r_c\} \times S^2_{\phi_*, \theta}
\end{align*}
are called the \emph{future event horizon} and \emph{future cosmological horizon}, respectively.
Note that the real analytic Killing vector fields $\d_t$ and $\d_\phi$, in Boyer-Lindquist coordinates, extend to real analytic Killing vector fields $\d_{t_*}$ and $\d_{\phi_*}$ on $(M_*, g_*)$.

We will consider wave equations on complex tensors.
Fixing $r, s \in \N_0$, let $\T_r^s\U$ denote the complex $(r, s)$-tensors on an open subset $\U \subset M_*$ and let $\n$ denote the Levi-Civita connection acting on $\T_r^s\U$.
We let $C^\infty(\T_r^s\U)$ and $C^\o(\T_r^s\U)$ denote the smooth and real analytic complex tensor fields, respectively.
Let $P$ be a wave operator, i.e.\ is a linear differential operator with principal symbol given by the dual metric, i.e.\
\[
	P = - g^{\a\b} \n_\a \n_\b + \text{lower order terms}.
\]
More precisely, there are complex tensor fields
\begin{align*}
	A
		&: T^*\U \otimes \T_r^s\U \to \T_r^s\U, \\*
	B
		&: \T_r^s\U \to \T_r^s\U,
\end{align*}
such that
\[
	P = \n^*\n + A \circ \n + B.
\]
We consider solutions to wave equations $Pu = f$, where the coefficients $A$ and $B$ are invariant under the Killing vector fields $\d_{t_*}$ and $\d_{\phi_*}$. 
This is a natural assumption for geometric wave equations, where $A$ and $B$ are typically given by curvature expressions.
Our first main result is the following:
\begin{thm} \label{mainthm: K and K-dS}
Let $(M_*, g_*)$ be the subextremal Kerr(-de Sitter) spacetime, extended real analytically over the future event horizon (and future cosmological horizon if $\Lambda > 0$).
Assume that 
\begin{itemize}
	\item $A$ and $B$ are real analytic in $M_*$,
	\item $\L_{\d_{t_*}}A = \L_{\d_{\phi_*}}A = 0$ and $\L_{\d_{t_*}}B = \L_{\d_{\phi_*}}B = 0$ in $M_*$.
\end{itemize}
If $u \in C^\infty(\T_r^sM_*)$ satisfies
\begin{enumerate}[label=(\roman*)]
	\item $Pu \in C^\o(\T_r^sM_*)$, \label{cond: first}
	\item $\L_{\d_{t_*}} u = - i \sigma u$ for some $\sigma \in \C$, \label{cond: second}
	\item $\L_{\d_{\phi_*}} u = - i k u$ for some $k \in \Z$, \label{cond: third}
\end{enumerate}
then $u \in C^\o(\T_r^sM_*)$.
\end{thm}

Smooth tensor fields satisfying \ref{cond: second} and \ref{cond: third} in Theorem \ref{mainthm: K and K-dS} and $Pu = 0$ are called \emph{quasinormal modes}.
For functions, these conditions are equivalent to assuming that
\begin{equation} \label{eq: joint modes functions}
	u(t_*, r, \phi_*, \theta) 
		= e^{- i(\sigma t_* + k \phi_*)}v(r, \theta),
\end{equation}
which is perhaps the more common way to express quasinormal modes. 

Combining Theorem \ref{mainthm: K and K-dS} with the Fredholm theory developed by the second author in \cite{V2013} and \cite{V2018} (see also \cites{VZ2000, V2019}) and by both authors in \cite{PV2021}, we deduce our second main result, where we consider quasinormal modes \emph{only} with respect to the Killing vector symmetry
\begin{equation} \label{eq: first new KVF}
	\d_{t_*} + \frac a{r_0^2 + a^2} \d_{\phi_*},
\end{equation}
where $r_0 \in (r_e, r_c)$ is the unique point such that $\mu'(r_0) = 0$, as opposed to modes with respect to both $\d_{t_*}$ and $\d_{\phi_*}$ separately (as in Theorem \ref{mainthm: K and K-dS}).
Concretely, this means that quasinormal modes are supposed to satisfy
\[
	\L_{\d_{t_*} + \frac a{r_0^2 + a^2}\d_{\phi_*}} u = - i \sigma u.
\]
For solutions to linear scalar wave equations on any subextremal Kerr-de Sitter spacetime, there is an asymptotic expansion in these quasinormal modes up to an exponentially decaying term \cite{PV2021}*{Thm.\ 1.5}.
This extends the result of \cite{V2013}, by removing restrictions on the angular momentum.

For the Fredholm theory to go through in the case $\Lambda = 0$, we will need the induced operator on the modes to be a scattering operator with self-adjoint (i.e.\ real) scattering principal symbol near spatial infinity in the sense of Melrose \cite{M1994}. 
Let us use the convention that if $\Lambda = 0$, then $r_0 = \infty$, giving the standard notion of quasinormal modes on the Kerr spacetime.
This amounts to making appropriate decay assumptions on $A$ and $B$:

\begin{thm} \label{mainthm: K and K-dS improved}
Let $(M_*, g_*)$ be the subextremal Kerr(-de Sitter) spacetime, extended real analytically over the future event horizon (and future cosmological horizon if $\Lambda > 0$).
Assume that 
\begin{itemize}
	\item $A$ and $B$ are real analytic in $M_*$,
	\item $\L_{\d_{t_*}}A = \L_{\d_{\phi_*}}A = 0$ and $\L_{\d_{t_*}}B = \L_{\d_{\phi_*}}B = 0$ in $M_*$.
\end{itemize}
If $u \in C^\infty(\T_r^sM_*)$ satisfies
\begin{enumerate}[label=(\roman*)]
	\item $Pu = 0$,
	\item $\L_{\d_{t_*} + \frac{a}{r_0^2 + a^2}\d_{\phi_*}} u = - i \sigma u$ for some $\sigma \in \C$, \label{cond: second improved}
     \item in case $\Lambda = 0$ we also assume that $\Im \sigma \geq 0$ and 
     \begin{itemize}
    		\item if $\Im \sigma > 0$, then assume that $A, B \in \O_\infty\left(r^{-\e}\right)$ and $u|_{t_* = 0} \in \mathcal S'$,
     	\item if $\sigma \in \R \backslash \{0\}$, then assume that $P - P^* \in \O_\infty\left(r^{-1-\e}\right)$ and ${A, B \in \O_\infty(r^{-\e})}$ and that $u|_{t_* = 0} \in r^{\frac12 - \e}L^2$,
     	\item if $\sigma = 0$, then assume that $A \in \O_\infty\left(r^{-1 - \e}\right)$ and $B = \O_\infty\left(r^{-2 - \e}\right)$ and $u|_{t_* = 0} \in \mathcal S'$,
     \end{itemize}
     for some $\e > 0$,
\end{enumerate}
then $u \in C^\o(\T_r^sM_*)$.
\end{thm}

Here we used the notation $\mathcal S'$ for tempered distributions and the notation ${F \in \O_\infty(r^\a)}$ for a complex tensor field $F$ if and only if for all $k \in \N_0$, there is a constant $C_k > 0$, such that
\[
	\abs{\n^k F} \leq C_k r^{\a-k},
\]
where $\abs{\cdot}$ is the positive definite norm on complex tensors induced from the Euclidean metric $\md t^2 + \md r^2 + r^2g_{S^2}$.
The notation $Q \in \O_\infty(r^\a)$ for a differential operator $Q$ means that the coefficients of $Q$ are in $\O_\infty(r^\a)$.

\begin{remark}
In the case when $\Lambda = 0$, one could weaken the assumptions on $u$, $A$ and $B$ at spatial infinity in various ways and still get a Fredholm problem following the arguments of \cite{V2018}.
Indeed, the natural condition on $u|_{t_* = 0}$ in the case $\sigma \in \R \backslash \{0\}$ is formulated microlocally in terms of variable order Sobolev spaces, c.f.\ \cite{V2018}*{Prop.\ 5.28}. 
Moreover, the threshold growth $r^{\frac12}$ could be adjusted depending on $A$ and $B$, to allow for more general coefficients, see \cite{V2018}*{Sec.\ 5.4.8}.
We restrict for simplicity to this setting.
\end{remark}

The restriction $\Im \sigma\geq 0$ for Kerr spacetimes is due to the
lack of a directly applicable Fredholm theory for the Fourier conjugated (in $-t_*$) operators in this case, though alternatives are still available for studying these resonances.
For functions, the condition \ref{cond: second improved} in Theorem \ref{mainthm: K and K-dS improved} is equivalent to assuming that
\[
	u(t_*, r, \phi_*, \theta) = e^{-i\sigma t_*}w(r, \phi_*, \theta),
\]
which should be compared with equation \eqref{eq: joint modes functions} above.

In the special case when $a = 0$, the Kerr(-de Sitter) spacetime simplifies to the Schwarzschild(-de Sitter) spacetime.
In this case, Theorem \ref{mainthm: K and K-dS} and Theorem \ref{mainthm: K and K-dS improved} can be immediately deduced from the framework developed by Galkowski-Zworski in \cite{GZ2020} as follows:
Wave equations for modes with respect to \eqref{eq: first new KVF} reduce in the coordinate system $(t_*, r,  \phi_*, \theta)$ to a Keldysh type operator, exactly of the type studied in \cite{GZ2020}. 
Galkowski-Zworski prove in \cite{GZ2020}*{Thm.~1} (generalizing \cite{Zu2017}*{Thm.~1.3}) the analytic hypoellipticity of such operators, thus proving the real analyticity of quasinormal modes when $a = 0$.
In fact, if $a = 0$, the argument goes through without assuming that the coefficients $A$ and $B$ are invariant under $\d_{\phi_*}$.
Due to the rotation in the Kerr(-de Sitter) spacetime when $a \neq 0$, this argument does not go through immediately.
The key to be able to apply the analytic hypoellipticity theory by Galkowski-Zworski to the case $a \neq 0$ is the main new idea of this paper and is described in the next subsection.

\subsection{Non-degenerate Killing horizons} \label{sec: non-deg horizon}

By checking the formula \eqref{eq: metric at horizons} for the extended metric $g_*$, one observes that the Killing vector field
\begin{equation} \label{eq: new KVF}
	\d_{t_*} + \frac a{r_0^2 + a^2} \d_{\phi_*},
\end{equation}
where $r_0 \in (r_e, r_c)$ is the unique point such that $\mu'(r_0) = 0$, is \emph{lightlike} at the horizons if and only if $a = 0$.
This turns out to be exactly why the modes with respect to \eqref{eq: new KVF} satisfy the useful Keldysh type equation if and only if $a = 0$.
In the Kerr(-de Sitter) spacetime, the Killing vector fields
\begin{align}
	& \d_{t_*} + \frac a{r_e^2 + a^2} \d_{\phi_*}, \label{eq: horizon KVF e} \\*
	\text{and } & \d_{t_*} + \frac a{r_c^2 + a^2} \d_{\phi_*} \text{ (if } \Lambda > 0), \label{eq: horizon KVF c}
\end{align}
are lightlike at the horizons $\H_e^+$ and $\H_c^+$ (if $\Lambda > 0$), respectively.
We show that the mode solutions with respect to \emph{these} Killing vector fields satisfy equations which are almost of Keldysh type.
More precisely, the bicharacteristics associated to the mode equation will have the \emph{radial point structure} assumed by Galkowski-Zworski in their analytic hypoellipticity result \cite{GZ2020}*{Thm.\ 2}.
Now, if $u$ satisfies the assumption of Theorem \ref{mainthm: K and K-dS}, then 
\begin{equation} \label{eq: joint modes}
	\L_{\d_{t_*} + \frac a{r_e^2 + a^2} \d_{\phi_*}} u 
		= - i \left(\sigma + \frac a{r_e^2 + a^2}k\right) u,
\end{equation}
and similarly with $r_e$ replaced by $r_c$.
This shows that $u$ is a mode solution with respect to both Killing vector fields \eqref{eq: horizon KVF e} and \eqref{eq: horizon KVF c}.
The analytic hypoellipticity result by Galkowski-Zworski thus shows that $u$ is real analytic near the horizons $\H^+_e$ and $\H^+_c$ (if $\Lambda > 0$).
This is the main step in the proof of Theorem \ref{mainthm: K and K-dS}, the rest follows by standard propagation of real analyticity for wave equations and analytic hypoellipticity of elliptic equations (c.f.\ \cite{M2002}*{Chapter 4}).

In fact, this method is not specific to the Kerr(-de Sitter) spacetime, but turns out to work for any Killing horizon in any real analytic spacetime, assuming the surface gravity of the Killing horizon is nowhere vanishing.
Assume therefore that $(M,g)$ is a real analytic spacetime, i.e.\ a time-oriented Lorentzian manifold, of dimension $n + 1 \geq 2$, with sign convention $(-, +, \hdots, +)$ and with a real analytic lightlike hypersurface $\H \subset M$.
We assume in particular that the metric $g$ is real analytic.

\begin{definition}\label{def: horizon Killing field}
A real analytic Killing vector field $W$ on $M$, such that $W|_\H$ is lightlike and tangent to $\H$, is called a \emph{horizon Killing vector field} with respect to $\H$.
\end{definition}
\noindent
For each Killing horizon $\H$ and horizon Killing vector field $W$, it is straightforward to check that
\begin{equation} \label{eq: kappa}
	\n_W W|_\H
		= \kappa W|_\H,
\end{equation}
for a real analytic function $\kappa: \H \to \R$ such that $W|_\H \kappa = 0$.
\begin{definition}
Given a Killing horizon $\H$ and a horizon Killing vector field $W$, the \emph{surface gravity} is the real analytic function $\kappa$ defined in \eqref{eq: kappa}.
\end{definition}
\noindent
The key assumption to prove real analyticity of quasinormal modes is that the surface gravity $\kappa$ is nowhere vanishing.
All horizons in Kerr(-de Sitter) spacetimes have surface gravity proportional to $\mu'$ at the horizons, where $\mu$ was defined in \eqref{eq: mu} (c.f.\ \textbf{Step 1} in the proof of Theorem \ref{mainthm: K and K-dS}).
This is the reason our result only applies to subextremal Kerr(-de Sitter) spacetimes, since subextremality makes sure that $\mu'$ does not vanish at the roots of $\mu$, i.e.\ at the horizons.
\begin{remark} \label{rmk: nowhere vanishing surface gravity}
We note in Lemma \ref{le: surface gravity} that if
\begin{equation} \label{eq: constant surface gravity condition}
	\Ric(X, W)|_\H
		= 0
\end{equation}
for all $X \in T\H$ and $\H$ is connected, then the surface gravity $\kappa$ is constant.
In practice, the condition \eqref{eq: constant surface gravity condition} is often satisfied. 
Indeed, it is for example satisfied if the spacetime satisfies the Einstein equation with a cosmological constant of any sign or if the spacetime satisfies the dominant energy condition (c.f.\ \cite{P2018}*{Rmk.~1.16}).
In case $\kappa$ is constant, we get a dichotomy of \emph{non-degenerate} Killing horizons, where $\kappa \neq 0$, and \emph{degenerate} Killing horizons, where $\kappa = 0$.
\end{remark}
As in the previous subsection, we fix $r, s \in \N_0$ and consider linear wave operators on complex $(r, s)$-tensors $\T_r^sM$ and write
\[
	P = \n^*\n + A \circ \n + B
\]
with complex tensor fields
\begin{align*}
	A
		&: T^*M \otimes \T_r^sM \to \T_r^sM, \\
	B
		&: \T_r^sM \to \T_r^sM.
\end{align*}
Our third main result in this paper is the following theorem:
\begin{thm} \label{mainthm: general}
Assume that
\begin{itemize}
\item $(M,g)$ is a real analytic spacetime,
\item $\H \subset M$ is a real analytic lightlike hypersurface,
\item $W$ is a real analytic horizon Killing vector field, 
\item the surface gravity $\kappa$ is nowhere vanishing,
\item $A$ and $B$ are real analytic and $\L_W A = 0$ and $\L_W B = 0$ on $M$.
\end{itemize} 
If $u \in C^\infty(T_r^sM)$ satisfies
\begin{enumerate}[label=(\roman*)]
	\item $Pu \in C^\o(\T_r^sM)$,
	\item $\L_W u = - i \sigma u$ for some $\sigma \in \C$, \label{cond: general second}
\end{enumerate}
then there is an open subset $\U \supset \H$, such that $u \in C^\o(T_r^s\U)$.
\end{thm}
Note that all assumptions in Theorem \ref{mainthm: general} are local.
As explained above, we will apply Theorem \ref{mainthm: general} with
\[
	W = \d_{t_*} + \frac a{r_e^2 + a^2} \d_{\phi_*},
\]
and with
\[
	W = \d_{t_*} + \frac a{r_c^2 + a^2} \d_{\phi_*},
\]
if $\Lambda > 0$, which will prove the main step in Theorem \ref{mainthm: K and K-dS} and Theorem \ref{mainthm: K and K-dS improved}, namely the real analyticity near the horizons.

Our methods require the existence of a horizon Killing vector field.
This allows to reduce the wave equation for the modes to the useful (almost) Keldysh form.
Surprisingly, a horizon Killing vector field is quite often guaranteed to exist in vacuum spacetimes with horizons.
Proving the existence of a horizon Killing vector field has been the central tool in various black hole uniqueness results for the subextremal Kerr spacetime. 
This line of argument was pioneered by Hawking, who showed that stationary real analytic vacuum black holes with a non-degenerate event horizon necessarily admit a horizon Killing vector field \cites{H1972, HE1973}.
This result was later generalized to higher dimensional analytic vacuum black holes by Hollands-Ishibashi-Wald \cite{HIW2007} and Moncrief-Isenberg \cite{MI2008}.

There is an analogous problem for compact (also called cosmological) Cauchy horizons in vacuum spacetimes.
A conjecture by Moncrief and Isenberg \cite{MI1983} states that any compact Cauchy horizon in a vacuum spacetime admits a horizon Killing vector field.
The existence of a horizon Killing vector field in that setting would prove that vacuum spacetimes with compact Cauchy horizons are non-generic, which would support the Strong Cosmic Censorship Conjecture in cosmology.
During the last decades, Moncrief and Isenberg have made important progress on their conjecture, assuming that the spacetime metric is real analytic \cites{MI1983,IM1985,MI2018}. 

Remarkably, the existence of a horizon Killing vector field does often not even rely on the real analyticity of the spacetime metric. 
Alexakis, Ionescu and Klainerman proved in \cite{AIK2010} (see also \cite{IonescuKlainerman2013}) an analogue of Hawking's theorem, showing the existence of a horizon Killing vector field in a neighbourhood of any bifurcate horizon in \emph{smooth} vacuum spacetimes, as opposed to real analytic.
This result has been central in their approach to prove uniqueness of subextremal Kerr black holes \cites{AIK2010_2, AIK2014} in the smooth setting.
For compact Cauchy horizons in smooth vacuum spacetimes, as opposed to real analytic, a horizon Killing vector field has been shown to exist by Petersen in \cite{P2019}, assuming that the surface gravity is a non-zero constant (extending \cites{FRW1999, P2018, PR2018}).
The assumption on constant surface gravity has recently been shown to be equivalent to a weak non-degeneracy assumption for compact Cauchy horizons in vacuum spacetimes, see \cite{BR2021} and \cite{GM2021}.

Though the above mentioned results mainly concern vacuum spacetimes without cosmological constant, one expects them to extend to the case of positive cosmological constant and electro-vacuum spacetimes as well (c.f.\ \cite{R2000}).
In conclusion, studying wave equations close to non-degenerate horizons (bifurcate or constant non-zero surface gravity), one might in quite wide generality be able to pass to modes with respect to the horizon Killing vector field and analyze the (almost) Keldysh type equation they are known to satisfy by the arguments in this paper.

\section{Suitable coordinates near non-degenerate Killing horizons}\label{sec:coordinates}

The first step towards proving Theorem \ref{mainthm: general} is to define appropriate coordinates near the lightlike hypersurface $\H$:

\begin{prop} \label{prop: suitable coordinates}
Assume the same as in Theorem \ref{mainthm: general}.
Then, for any $p \in \H$, there is a real analytic coordinate system $(x_0, \hdots, x_n)$, defined on an open neighborhood $\U \ni p$, such that
\begin{itemize}
\item $\d_{x_0} = W|_\U$,
\item $x_1$ is a defining function for $\U \cap \H$ (i.e.\ $\U \cap \H = x_1^{-1}(0)$ and $\md x_1|_{\U \cap \H} \neq 0$), 
\item the metric $g$ expressed in these coordinates satisfies
\begin{equation} \label{eq: metric at r= 0}
	g|_{x_1 = 0} 
		= 
	\begin{pmatrix}
		0 & 1 & 0 & \hdots & 0 \\
		1 & 0 & 0 & \hdots & 0 \\
		0 & 0 & g_{22}|_{x_1 = 0}  & \hdots & g_{2 n}|_{x_1 = 0}  \\
		\vdots & \vdots & \vdots & \ddots & \vdots \\
		0 & 0 & g_{n2}|_{x_1 = 0} & \hdots & g_{nn}|_{x_1 = 0}  \\
	\end{pmatrix}, 
\end{equation}
where
\begin{equation} \label{eq: submetric at r= 0}
	\begin{pmatrix}
		g_{22}|_{x_1 = 0}  & \hdots & g_{2 n}|_{x_1 = 0}  \\
		\vdots & \ddots & \vdots \\
		g_{n2}|_{x_1 = 0}  & \hdots & g_{nn}|_{x_1 = 0}  \\
	\end{pmatrix},
\end{equation}
is positive definite and
\[
	\d_1 g_{00}|_{x_1 = 0} = - 2\kappa,
\]
where $\kappa$ is the (nowhere vanishing) surface gravity.
\end{itemize}
\end{prop}

\begin{remark}
These coordinates are essentially the \emph{Gaussian null coordinates} introduced by Moncrief-Isenberg in \cite{MI1983}, with the extra condition that $\d_0$ is the horizon Killing vector field restricted to an open neighborhood.
(This is precisely what is obtained \emph{a posteriori} after the construction of the horizon Killing vector field in \cite{MI1983}.)
\end{remark}

\begin{example} \label{ex: Misner}
The simplest example of a spacetime satisfying all our assumptions is $M = \R^{n+1}$, equipped with the real analytic Misner metric
\[
	g = 2\md x_1\md x_0 + x_1 \md x_0^2 + \sum_{j = 2}^n (\md x^j)^2,
\]
where $\H = \{x_1 = 0\}$, $W = \d_0$ and surface gravity $\kappa = - \frac12$.
\end{example}

\begin{example} \label{ex: Kerr-de Sitter}
In fact, even in the subextremal Kerr(-de Sitter) spacetime, one can easily choose coordinates which almost satisfy the conditions in Proposition \ref{prop: suitable coordinates}, with one (insignificant) difference. 
To define these, it will be convenient to introduce an intermediate coordinate system, which will only be defined near one of the horizons.
Let us start with the event horizon.
In terms of Boyer-Lindquist coordinates, define
\begin{align*}
	\tilde t_*
		&:= t - \tilde \Phi(r), \\
	\tilde \phi_*
		&:= \phi - \tilde \Psi(r),
\end{align*}
where $\tilde \Phi$ and $\tilde \Psi$ satisfy
\begin{align*}
		\tilde \Phi'(r)
			&= - b \frac{r^2 + a^2}{\mu(r)}, \\
		\tilde \Psi'(r)
			&= - b \frac a{\mu(r)},
\end{align*}
near $r = r_e$.
This commonly used analytic coordinate system $(\tilde t_*, r, \tilde \phi_*, \theta)$ is defined near the future event horizon.
Choose now the coordinates
\begin{align*}
	x_0
		&= \tilde t_*, \\
	x_1
		&= r - r_e, \\
	x_2
		&= \tilde \phi_* - \frac a {r_e^2 + a^2}\tilde t_*, \\
	x_3
		&= \theta,
\end{align*}
from which we get
\[
	\d_{x_0} = \d_{\tilde t_*} + \frac a{r_e^2 + a^2}\d_{\tilde \phi_*}.
\]
Defining
\[
	\psi(x_3) := b \frac{r_e^2 + a^2}{r_e^2 + a^2 \cos^2(x_3)},
\]
one easily computes that the metric $g_*$ at the future event horizon is given by
\begin{equation} \label{eq: KdS metric at r= 0}
	\psi g_*|_{x_1 = 0} 
		=
	\begin{pmatrix}
		0 & 1 & 0 & 0 \\
		1 & 0 & 0 & 0 \\
		0 & 0 & {g_*}_{22}|_{x_1 = 0} & 0  \\
		0 & 0 & 0 & {g_*}_{33}|_{x_1 = 0} \\
	\end{pmatrix},
\end{equation}
in these coordinates, where ${g_*}_{22}|_{x_1 = 0}, {g_*}_{33}|_{x_1 = 0} > 0$.
Moreover, we have
\[
	\d_1 (\psi {g_*}_{00})|_{x_1 = 0} = - 2 \kappa_e,
\]
where the surface gravity $\kappa_e$ is given by
\[
	\kappa_e 
		= \frac{\mu'(r_e)}{2 b \left(r_e^2 + a^2\right)} 
		> 0,
\]
c.f.\ the computation in Step 1 in the proof of Theorem \ref{mainthm: K and K-dS}.
These coordinates coincide with the coordinates in Proposition \ref{prop: suitable coordinates}, up to the multiplication by the positive conformal factor $\psi$. 
Since conformal changes of the geometry only reparametrize the lightlike geodesics, $\psi$ is irrelevant for the analysis.
However, it is of course natural to construct the coordinates in Proposition \ref{prop: suitable coordinates} without a conformal factor. 
This would here correspond to changing $x_1$ to $\tilde x_1$ by solving the geodesic equation
\[
	\n_{\d_{\tilde x_1}}\d_{\tilde x_1} = 0, \quad \d_{\tilde x_1}|_{x_1 = 0} = \psi \d_{x_1}|_{x_1 = 0},
\]
and changing the remaining coordinates $x_j$ to $\tilde x_j$ by demanding that
\[
	[\d_{\tilde x_1}, \d_{\tilde x_j}] = 0, \quad \d_{\tilde x_j}|_{x_1 = 0} = \d_{x_j}|_{x_1 = 0}.
\]
In this new coordinate system, we get precisely the conditions in Proposition \ref{prop: suitable coordinates}.
One analogously constructs similar coordinates near the future cosmological horizon.
\end{example}

\begin{proof}[Proof of Proposition \ref{prop: suitable coordinates}]
Let first $(x_0, x_2, \hdots, x_n)$ be real analytic coordinates in an open neighborhood $\V \subseteq \H$ of $p$, such that
\[
	\d_0 = W|_\V.
\]
Now let $L$ be the unique lightlike real analytic vector field (transversal to $\H$) along $\V$ such that
\begin{equation} \label{eq: def L}
	g(L, L)|_\V = g(L, \d_j)|_\V  = 0, \quad g(L, \d_0)|_\V = 1
\end{equation}
for $j = 2, \hdots, n$.
Define now the real analytic coordinate $x_1$ in an open neighborhood $\U \subset M$ of $p$, such that $\V = \U \cap \H$, by solving the geodesic equation in direction of $L$, i.e.\ we solve
\begin{align*}
	\n_{\d_1}\d_1
		&= 0, \\
	\d_1|_\V
		&= L
\end{align*}
and set $x_1 = 0$ at $\H$.
It follows that $x_1$ is a defining function for $\U \cap \H$.
We also extend the other coordinates to $\U$ by demanding that
\[
	[\d_1, \d_0] = [\d_1, \d_j] = 0
\]
in $\U$, for $j = 2, \hdots, n$.
The inverse function theorem for real analytic functions implies that this forms a coordinate system.

We now show that $\d_0 = W|_\U$.
Recall first that $\d_0|_{x_1 = 0} = W|_{x_1 = 0}$.
By uniqueness of ODE and since $[\d_0, \d_1] = 0$, it suffices to show that $[\d_1, W|_\U] = 0$.
This is equivalent to $W$ leaving the integral curves of $\d_1$ invariant.
Since $W$ is a Killing vector field and the integral curves of $\d_1$ are geodesics, it thus suffices to prove that the initial velocity $\d_1|_{x_1 = 0}$ of the geodesics are invariant under $W$, i.e.\ that
\[
	[W, \d_1]|_{x_1 = 0} = \left( \n_W\d_1 - \n_{\d_1}W \right)|_{x_1 = 0} = 0.
\]
Since $W|_{x_1 = 0} = \d_0|_{x_1 = 0}$, it follows that $\n_W\d_1|_{x_1 = 0} = \n_{\d_0}\d_1|_{x_1 = 0}$ and therefore, since $[\d_0, \d_1]|_{x_1 = 0} = 0$, it suffices to prove that $\n_{\d_1}W|_{x_1 = 0} = \n_{\d_1}\d_0|_{x_1 = 0}$.
Using that $W$ is a Killing vector field, we observe that
\begin{align*}
	g(\n_{\d_1} W, \d_1)
		&= \frac12 \L_W g(\d_1, \d_1) \\
		&= 0.
\end{align*}
We also have
\begin{align*}
	g(\n_{\d_1} \d_0, \d_1)|_{x_1 = 0}
		&= g([\d_1, \d_0], \d_1)|_{x_1 = 0} - g(\n_{\d_0} \d_1, \d_1)|_{x_1 = 0} \\
		&= - \frac12 \d_0 g(\d_1, \d_1)|_{x_1 = 0} \\
		&= 0,
\end{align*}
hence
\[
	g(\n_{\d_1} W, \d_1)|_{x_1 = 0}
		= 0 
		= g(\n_{\d_1} \d_0, \d_1)|_{x_1 = 0}.
\]
Moreover, for all $j = 0, 2, \hdots, n$, we have
\begin{align*}
	g(\n_{\d_1}W, \d_j)|_{x_1 = 0}
		&= \L_Wg(\d_1, \d_j)|_{x_1 = 0} - g(\n_{\d_j} W, \d_1)|_{x_1 = 0} \\*
		&= - g(\n_{\d_j} \d_0, \d_1)|_{x_1 = 0} \\*
		&= - g(\n_{\d_0} \d_j, \d_1)|_{x_1 = 0} \\*
		&= - \d_0 g(\d_j, \d_1)|_{x_1 = 0} + g(\d_j, \n_{\d_0}\d_1)|_{x_1 = 0} \\*
		&= g(\n_{\d_1} \d_0, \d_j)|_{x_1 = 0},
\end{align*}
where we have used that $g(\d_j, \d_1)|_{x_1 = 0}$ is constant by \eqref{eq: def L}.
This shows that 
\[
	\n_{\d_1} W|_{x_1 = 0} = \n_{\d_1} \d_0 |_{x_1 = 0}.
\]
Taken together, this shows our claim that $\d_0 = W|_\U$.

It is now clear that the metric has the form \eqref{eq: metric at r= 0} at $x_1 = 0$ and that the part \eqref{eq: submetric at r= 0} is positive definite.
Using \eqref{eq: kappa}, we compute
\begin{align*}
	\d_1 g_{00}|_{x_1 = 0}
		&= 2 g(\n_{\d_1}\d_0, \d_0)|_{x_1 = 0} \\*
		&= 2\d_0 g(\d_1, \d_0)|_{x_1 = 0} - 2 g(\d_1, \n_{\d_0}\d_0)|_{x_1 = 0} \\*
		&= - 2 g(\d_1, \n_WW)|_{x_1 = 0} \\*
		&= - 2 \kappa g(\d_1, W)|_{x_1 = 0} \\*
		&= - 2 \kappa.
\end{align*}
This finishes the proof.
\end{proof}

\section{Real analyticity near general horizons}\label{sec:general}

The goal of this section is to prove Theorem \ref{mainthm: general}.
In order to explain the idea, let us start by discussing the following example:
\begin{example} \label{ex: Misner modes}
The d'Alembert operator in Example \ref{ex: Misner} is given by
\[
	\Box = \d_1 \left(x_1 \d_1 - 2\d_0\right) - \sum_{j=2}^n \d_j^2.
\]
The condition \ref{cond: general second} in Theorem \ref{mainthm: general} is that
\[
	u(x_0, \hdots, x_n) = e^{- i \sigma x_0} v(x_1, \hdots, x_n).
\]
Such a mode solution to $\Box u = 0$ must satisfy the reduced equation
\[
	\d_1(x_1 \d_1 v) - \sum_{j=2}^n \d_j^2 v + 2 i\sigma \d_1 v= 0.
\]
This is a Keldysh type equation on the quotient space
\[
	\R^{n+1}/{\sim} = \R^n,
\]
and \cite{GZ2020}*{Thm.\ 1} implies that $v$ and hence $u$ is real analytic.
\end{example}

The proof of Theorem \ref{mainthm: general} is a generalization of the argument in Example \ref{ex: Misner modes}:

\begin{proof}[Proof of Theorem \ref{mainthm: general}]
Shrinking $\U$ if necessary, we can write the coordinates from Proposition \ref{prop: suitable coordinates} as
\[
	(x_0, \hdots, x_n): \U \to (-\e, \e)_{x_0} \times (-\de, \de)_{x_1} \times K_{x_2, \hdots, x_n} \subset \R^{n+1},
\]
where $K \subset \R^{n-1}$ is an open relatively compact subset and $\e, \de > 0$ are sufficiently small.
Since
\[
	\d_0 = W|_\U
\]
is a Killing vector field, we would like to eventually reduce $P$ in the $x_0$-direction.
For this, we first set
\[
	\V := \U/{\sim},
\]
where $p \sim q$ if and only if
\[
	(x_1(p), \hdots, x_n(p)) = (x_1(q), \hdots, x_n(q)),
\]
i.e.\ only $x_0(p)$ and $x_0(q)$ may differ.
The induced coordinates on the quotient space are
\[
	(x_1, \hdots, x_n): \V \to (-\de, \de)_{x_1} \times K_{x_2, \hdots, x_n},
\]
i.e.\ we have \enquote{dropped} the $x_0$-coordinate.

The complex $(r, s)$-tensors on $\U$ are complex linear combinations of basis elements of the form
\[
	e_\I := \d_{i_0} \otimes \hdots \otimes \d_{i_r} \otimes \md x_{j_0} \otimes \hdots \otimes \md x_{j_s},
\]
where $\I := (i_1, \hdots, i_r, j_1, \hdots, j_s)$, and we of course have
\[
	\L_{\d_0} e_\I = 0.
\]
Let us define $f := Pu$ and write
\[
	u = \sum_{\I}u_\I e_\I, \quad f = \sum_{\I}f_\I e_\I.
\]
Since $\d_0 = W|_\U$ is a Killing vector field, we note that
\[
	[\L_{\d_0}, \n] = 0,
\]
and by the assumption in Theorem \ref{mainthm: general}, we know that
\[
	\L_{\d_0}A = \L_W A = 0, \quad \L_{\d_0}B = \L_W B = 0.
\]
It thus follows that the wave equation $Pu = f$, restricted to the subset $\U$, can be written as a linear system of scalar wave equations
\begin{equation} \label{eq: system of equations}
	\sum_{\a,\b = 0}^n -g^{\a \b} \d_\a \d_\b u_\I + \sum_{\gamma = 0}^n\sum_{J}\mathcal A_{\I, \gamma}^\J\d_\gamma u_\J + \sum_\J \mathcal B_\I^\J u_\J = f_\I,
\end{equation}
for each $\I := (i_1, \hdots, i_r, j_1, \hdots, j_s)$, where the coefficients
\[
	g^{\a \b}, \quad \mathcal A_{\I, \gamma}^\J, \quad \mathcal B_\I^\J
\]
are \emph{independent} of $x_0$.
By the mode condition \ref{cond: general second}, we note that
\[
	\d_0 u_\I = - i \sigma u_\I, \quad \d_0 f_\I = - i \sigma f_\I,
\]
which implies that 
\[
	u_\I = e^{- i \sigma x_0} u_\I|_{x_0 = 0}, \quad f_\I = e^{- i \sigma x_0} f_\I|_{x_0 = 0}.
\]
Inserting this into \eqref{eq: system of equations} gives a new system of equations
\[
	\sum_{i,j = 1}^n - g^{ij} \d_i \d_j u_\I|_{x_0 = 0} + \sum_{k = 1}^n\sum_J \mathcal C_{\I, k}^\J\d_k u_\J|_{x_0 = 0} + \sum_\J \mathcal D_\I^\J u_\J|_{x_0 = 0}
		= f_\I |_{x_0 = 0},
\]
where the new coefficients $\mathcal C_{\I, k}^\J$ and $\mathcal D_\I^\J$ are independent of $x_0$.
Note also that the sums now exclude derivatives in $x_0$.
We have thus shown that the equation $Pu = f$ is equivalent to a system of equations 
\[
	\hat P u|_{x_0 = 0} = f|_{x_0 = 0}
\]
on the quotient space
\[
	\V = \U/{\sim},
\]
where the principal symbol of $\hat P$ is 
\begin{equation} \label{eq: principal symbol}
	\p(x_1, \hdots, x_n, \xi_1, \hdots, \xi_n) \Id,
\end{equation}
where $\Id$ is the identity matrix and
\begin{equation*}
	\p(x_1, \hdots, x_n, \xi_1, \hdots, \xi_n) 
		:= \sum_{i,j = 1}^n g(x_1, \hdots, x_n)^{ij}\xi_i \xi_j
\end{equation*}
for any $(x_1, \hdots, x_n, \xi_1, \hdots, \xi_n) \in T^* \V$.

This is where the information about the metric $g$ in Proposition \ref{prop: suitable coordinates} becomes useful.
We claim that first that
\begin{equation} \label{eq: conormal bundle}
	\{\p = 0\} \cap \{x_1=0\} = N^*\{x_1=0\},
\end{equation}
where $ N^*\{x_1=0\}$ denotes the conormal bundle of the horizon $\{x_1=0\}$.
In order to compute the components $g^{ij}$, for $i,j = 1, \hdots, n$, we first need to invert the \emph{full} matrix of metric components.
By Proposition \ref{prop: suitable coordinates}, we conclude that 
\begin{equation} \label{eq: inverse metric at x_1= 0}
	g^{\a\b}|_{\U \cap \{x_1 = 0\}} 
		= 
	\begin{pmatrix}
		0 & 1 & 0 & \hdots & 0 \\
		1 & 0 & 0 & \hdots & 0 \\
		0 & 0 & g^{22}|_{x_1 = 0}  & \hdots & g^{2 n}|_{x_1 = 0}  \\
		\vdots & \vdots & \vdots & \ddots & \vdots \\
		0 & 0 & g^{n2}|_{x_1 = 0} & \hdots & g^{nn}|_{x_1 = 0}  \\
	\end{pmatrix}
\end{equation}
for $\a, \b = 0, \hdots, n$.
The components appearing in \eqref{eq: principal symbol} are given, at $x_1 = 0$, by
\[
	g^{ij}|_{x_1 = 0} = 	
	\begin{pmatrix}
		0 & 0 &  \hdots & 0 \\
		0 & g^{22}|_{x_1 = 0} & \hdots & g^{2 n}|_{x_1 = 0} \\
		\vdots & \vdots & \ddots & \vdots \\
		0 & g^{n2}|_{x_1 = 0} & \hdots & g^{nn}|_{x_1 = 0} \\
	\end{pmatrix}.
\]
Since the matrix
\[
	\begin{pmatrix}
		g^{22}|_{x_1 = 0} & \hdots & g^{2 n}|_{x_1 = 0} \\
		\vdots & \ddots & \vdots \\
		g^{n2}|_{x_1 = 0} & \hdots & g^{nn}|_{x_1 = 0} \\
	\end{pmatrix}
\]
is positive definite by Proposition \ref{prop: suitable coordinates}, we have proven \eqref{eq: conormal bundle}.

By standard \emph{microlocal} analytic hypoellipticity at elliptic points in $T^*\V$, see for example \cite{M2002}*{Thm.\ 4.2.2 \& Exe.\ 4.6.4}, we hence conclude that $u_\I$ is microlocally real analytic everywhere at $x_1 = 0$ except potentially at the conormal bundle $N^* \{x_1 = 0\}$, i.e.\ the analytic wave front set at $x_1 = 0$ is contained in the conormal bundle.
We will show the real analyticity at the conormal bundle by applying \cite{GZ2020}*{Thm.\ 2}, which requires a computation of the Hamiltonian vector field $H_\p$ at $N^* \{x_1 = 0\}$.
For this, we first compute $\d_1 \p|_{N^* \{x_1 = 0\}}$.
At an arbitrary point
\[
	q := (0, x_2, \hdots, x_n, \xi_1, 0, \hdots, 0) \in N^* \{x_1 = 0\},
\]
using \eqref{eq: inverse metric at x_1= 0} and Proposition \ref{prop: suitable coordinates}, we compute
\begin{align*}
	\d_1\p|_q
		&= \d_1g^{11}|_q (\xi_1)^2 \\
		&= - \sum_{\a, \b = 0}^n g^{\a 1}(\d_1g_{\a\b})g^{\b1}|_q (\xi_1)^2 \\
		&= - \d_1g_{00}|_q (\xi_1)^2 \\
		&= 2 \kappa (\xi_1)^2.
\end{align*}
We may now compute the Hamiltonian vector field as
\begin{align*}
	H_\p|_q
		&=  \sum_{j = 1}^n (\d_{\xi_j}\p) \d_j|_q - (\d_j\p) \d_{\xi_j}|_q \\
		&= - (\d_1 \p) \d_{\xi_1}|_q \\
		&= - 2 \kappa (\xi_1)^2 \d_{\xi_1}|_q,
\end{align*}
where we recall that $\kappa$ is nowhere vanishing.
In particular
\[
	\md \p|_{N^* \{x_1 = 0\} \backslash \{0\}} \neq 0
\]
and 
\[
	H_\p|_{N^* \{x_1 = 0\} \backslash \{0\}} \parallel \xi \cdot \d_\xi,
\]
which means that the assumptions in  \cite{GZ2020}*{Thm.\ 2} are satisfied.
Note here that \cite{GZ2020}*{Thm.\ 2} is only proven for scalar valued wave equations, but the argument goes through line by line for systems of equations with a scalar principal symbol, as in our case.
Indeed, \cite{GZ2020}*{Thm.\ 2} relies on Haber’s normal form in \cite{H2014} {\em only for the principal symbol}. 
Thus, having scalar principal symbol suffices.
Hence \cite{GZ2020}*{Thm.\ 2} implies that $u_\I|_{t_* = 0}$ is microlocally real analytic also at the conormal bundle.
It follows that $u_\I|_{t_* = 0}$ is real analytic in an open subset containing $\{x_1 = 0\}$.
Consequently, $u_\I$ and therefore $u$ is real analytic in an open neighborhood containing $p$, which completes the proof.
\end{proof}

\section{Joint quasinormal modes}
\label{sec: proof of first main result}

We continue by proving the next main result of this paper:
\begin{proof}[Proof of Theorem \ref{mainthm: K and K-dS}]
Let us for simplicity restrict in this proof to the case of complex functions, as opposed to complex tensor fields of higher rank.
This will make the proof more transparent and avoid the technical details involved with working with system of equations.
All such technicalities are already present in the proof of Theorem \ref{mainthm: general} above. 
We thus consider functions of the form
\[
	u(t_*, r, \phi_*, \theta) = e^{-i(\sigma t_* + k\phi_*)}v(r, \theta),
\]
which are smooth on
\[
	\U = \R_{t_*} \times (r_e - \de, r_c + \de)_r \times S^2_{\phi_*, \theta} \subset M_*.
\]
We aim to prove that $u$ is real analytic on $\U$.

\textbf{Step 1: Real analyticity near the horizons.}
We would like to apply Theorem \ref{mainthm: general} with $\H = \H_{e/c}^+$ and therefore need to check that all assumptions of Theorem \ref{mainthm: general} are satisfied.
Firstly, the Kerr(-de Sitter) spacetime is a vacuum spacetime with a non-negative cosmological constant, so the dominant energy condition is clearly satisfied.
Moreover, the horizons $\H_{e/c}^+$ are real analytic lightlike hypersurfaces.
Secondly, the Killing vector fields
\[
	W_{e/c}
		:= \d_{t_*} + \frac a{r_{e/c}^2 + a^2} \d_{\phi_*}
\]
are clearly lightlike at $\H_{e/c}^+$, respectively.
Since $A$ and $B$ are invariant under $\d_{t_*}$ and $\d_{\phi_*}$, they are also invariant under $W_{e/c}$.
Further, the surface gravity $\kappa_{e/c}$ of the horizons is computed using the extended metric in \eqref{eq: metric at horizons} as follows:
\begin{align*}
	\d_r g_* \left( W_{e/c}, W_{e/c}\right)|_{r_{e/c}}
		&= 2 g_*\left( \n_{\d_r} W_{e/c}, W_{e/c}\right)|_{r_{e/c}} \\
		&= 2 g_*\left( \n_{W_{e/c}}\d_r, W_{e/c}\right)|_{r_{e/c}} \\
		&= - 2 g_*\left( \d_r, \n_{W_{e/c}} W_{e/c}\right)|_{r_{e/c}} \\
		&= - 2 \kappa_{e/c} \ g_*\left( \d_r, W_{e/c}\right)|_{r_{e/c}} \\
		&= \mp \frac {2\kappa_{e/c}}b \frac{r_{e/c}^2 + a^2 \cos^2(\theta)}{r_{e/c}^2 + a^2}.
\end{align*}
On the other hand, we have
\begin{align*}
	\d_r g_* \left( W_{e/c}, W_{e/c}\right)|_{r_{e/c}}
		&= - \d_r \left( \frac{\mu(r)}{b^2\left(r^2 + a^2 \cos^2(\theta)\right)} \left( \frac{r^2 + a^2 \cos^2(\theta)}{r_{e/c}^2 + a^2} \right)^2 \right)|_{r_{e/c}} \\*
		&\qquad + \d_r \left( \frac{c(\theta)\sin^2(\theta)}{b^2\left(r^2 + a^2 \cos^2(\theta)\right)}\left(a - a \frac{r^2 + a^2}{r_{e/c}^2 + a^2} \right)^2 \right)|_{r_{e/c}} \\*
		&= - \frac{\mu'(r_{e/c})}{b^2} \frac{r_{e/c}^2 + a^2 \cos^2(\theta)}{\left( r_{e/c}^2 + a^2 \right)^2}.
\end{align*}
The surface gravity of the horizons $\H_{e/c}^+$ is thus given by
\[
	\kappa_{e/c}
		= \pm \frac{\mu'(r_{e/c})}{2b (r_{e/c}^2 + a^2)}.
\]
Since this is non-zero, we conclude that $\H_{e/c}^+$ are non-degenerate Killing horizons with respect to $W_{e/c}$.
Finally, we compute that
\begin{align*}
	\L_{W_{e/c}} u 
		&= \L_{\d_{t_*}}u + \frac a{r_{e/c}^2 + a^2} \L_{\d_{\phi_*}}u \\*
		&= - i \sigma u - i \frac a{r_{e/c}^2 + a^2} k u \\*
		&= - i \left( \sigma + \frac a{r_{e/c}^2 + a^2} k \right) u,
\end{align*}
which shows that $u$ is a mode with respect to $W_{e/c}$.
We may therefore apply Theorem \ref{mainthm: general} and conclude that $u$ is real analytic in open neighborhoods of the horizons, which are invariant under $W_e$ and $W_c$, respectively.

\textbf{Step 2: Real analyticity in the domain of outer communication.}
We now prove real analyticity of $u$ in the open subset
\[
	\W := \R_{t_*} \times \left(r_e, r_c\right)_r \times S^2_{\phi_*, \theta}.
\]
(Recall that $r_c = \infty$ in the Kerr spacetime).
The Boyer-Lindquist coordinates $(t, r, \phi, \theta)$ are defined on this set and are convenient to work with.
Since
\[
	\d_t = \d_{t_*}|_\W, \quad \d_\phi = \d_{\phi_*}|_\W,
\]
the conditions \ref{cond: second} and \ref{cond: third} in Theorem \ref{mainthm: K and K-dS} imply that 
\[
	u(t, r, \phi, \theta) = e^{- i(\sigma t + k \phi)} w(r, \theta),
\]
so we can equally well consider the modes with respect to the Boyer-Lindquist coordinates.
The dual metric $G$ of $g$ in Boyer-Lindquist coordinates is
\begin{equation} \label{eq: BL dual metric}
\begin{split}
	(r^2 + a^2 \cos^2(\theta))G
		&= \mu(r)\d_r^2 + c(\theta)\d_\theta^2 + \frac{b^2}{c(\theta)\sin^2(\theta)}\left(a \sin^2(\theta)\d_t + \d_\phi\right)^2 \\
		&\qquad - \frac{b^2}{\mu(r)} \left( (r^2 + a^2)\d_t + a \d_\phi \right)^2.
\end{split}
\end{equation}
We begin by proving real analyticity of $w$ in the open subset
\[
	(r_e, r_c)_r \times (0, \pi)_\theta
\]
i.e.\ we leave out the north and the south pole of $S^2_{\phi, \theta}$ for the moment.
Since we have assumed that the coefficients of $P$ are independent of $t$ and $\phi$, the function $w$ satisfies an induced equation on $(r_e, r_c)_r \times (0, \pi)_\theta$, with principal part given by
\begin{equation} \label{eq: principal not north south}
	\frac1{r^2 + a^2 \cos^2(\theta)} \left(\mu(r)\d_r^2 + c(\theta)\d_\theta^2 \right).
\end{equation}
Since $\mu(r), c(\theta) > 0$ in this set, the induced equation for $w$ is elliptic with real analytic coefficients. 
Standard analytic hypoellipticity, see for example \cite{M2002}*{Thm.\ 4.2.2 \& Exe.\ 4.6.4}, therefore implies that $w$ is real analytic in $(r_e, r_c)_r \times (0, \pi)_\theta$ and hence $u$ is real analytic in
\[
	\R_t \times \left(r_e, r_c\right)_r \times S^1_\phi \times (0, \pi)_\theta.
\]

We now turn to show that $u$ is also real analytic at the north and south poles of $S^2_{\phi, \theta}$, i.e.\ at the limits $\theta = 0$ and $\theta = \pi$, still with $r \in (r_e, r_c)$.
Note that the expression \eqref{eq: principal not north south} does not extend smoothly to those points.
We now write 
\[
	u(t, r, \phi, \theta) = e^{- i\sigma t} z(r, \phi, \theta),
\]
i.e.\ the idea is to show real analyticity of
\[
	z(r, \phi, \theta) := e^{- i k\phi}w(r, \theta),
\]
which is smooth in
\[
	(r_e, r_c)_r \times S^2_{\phi, \theta}.
\]
Since the coefficients of $P$ are independent of $t$, we get an induced equation for $z$, with real analytic coefficients and principal part
\begin{equation} \label{eq: mode equation in t}
	\mu(r)\d_r^2 + c(\theta)\d_\theta^2 + \frac{b^2}{c(\theta)\sin^2(\theta)} \d_\phi^2 - \frac{b^2}{\mu(r)} a^2 \d_\phi^2.
\end{equation}
We claim that this operator is elliptic at $\theta = 0$ and $\theta = \pi$.
For this, we note that
\begin{align*}
	c(\theta)\d_\theta^2 + \frac{b^2}{c(\theta)\sin^2(\theta)} \d_\phi^2
		&= \left(c(\theta) - \frac{b^2}{c(\theta)}\right)\d_\theta^2 + \frac{b^2}{c(\theta)} \left(\frac1{\sin^2(\theta)} \d_\phi^2 + \d_\theta^2 \right) \\
		&= \frac1{c(\theta)} \left(c(\theta)^2 - b^2\right)\d_\theta^2 + \frac{b^2}{c(\theta)}G_{S^2} \\
		&= \frac1{c(\theta)} \left(\left(b - \frac{\Lambda a^2}3\sin^2(\theta) \right)^2 - b^2\right)\d_\theta^2 + \frac{b^2}{c(\theta)}G_{S^2} \\
		&= h(\theta) \sin^2(\theta)\d_\theta^2 + \frac{b^2}{c(\theta)}G_{S^2},
\end{align*}
for some function $h$, which extends real analytically to $S^2$ and where $G_{S^2}$ is the dual metric to the standard metric on $S^2$.
Since both $\sin^2(\theta)\d_\theta^2$ and $G_{S^2}$ extend real analytically to $S^2$, we can evaluate this expression at $\theta = 0$ or $\theta = \pi$ and conclude that \eqref{eq: mode equation in t} is simply
\begin{align*}
	\mu(r)\d_r^2 + b G_{S^2}
\end{align*}
at the north and the south pole of $S^2$.
Since $\mu(r) > 0$ for $r \in (r_e, r_c)$ and $b > 0$, we conclude that \eqref{eq: mode equation in t} is indeed elliptic there as well.
Again, standard real analytic hypoellipticity, as in for example \cite{M2002}*{Thm.\ 4.2.2 \& Exe.\ 4.6.4}, implies that $z$, and therefore $u$, is real analytic also at the north and the south pole if $r \in (r_e, r_c)$.
To sum up, we now know that $u$ is real analytic in the domain of outer communication and slightly beyond the horizons, i.e.\ in a region of the form
\[
	\R_{t_*} \times (r_e - \e, r_c + \e)_r \times S^2_{\phi_*, \theta}.
\]

\textbf{Step 3: The region beyond the horizons.}
It remains to prove real analyticity in the regions beyond the horizons (only the event horizon if $\Lambda = 0$).
Consider first the region 
\[
	\R_{t_*} \times \left(r_e - \de, r_e \right)_r \times S^2_{\phi_*, \theta},
\]
beyond the event horizon.
We may use Boyer-Lindquist coordinates also here, since this region does not intersect any horizon, where the coordinates would not be defined.
Let us again consider
\[
	z(r, \phi, \theta) := e^{- ik\phi}w(r, \theta),
\]
which by assumption is smooth in
\[
	\left(r_e - \de, r_e\right)_r \times S^2_{\phi, \theta}.
\]
In this set, we have $\mu(r) < 0$.
Again, the coefficients of $P$ are independent of $t$ and the principal part of the induced equation for $z$ can be read off from \eqref{eq: mode equation in t} to be
\[
		- \abs{\mu(r)}\d_r^2 + c(\theta)\d_\theta^2 + \left(\frac{b^2}{c(\theta)\sin^2(\theta)} + \frac{b^2}{\abs{\mu(r)}} a^2\right) \d_\phi^2.
\]
Since the operator
\[
	c(\theta)\d_\theta^2 + \left(\frac{b^2}{c(\theta)\sin^2(\theta)} + \frac{b^2}{\abs{\mu(r)}} a^2\right) \d_\phi^2.
\]
is elliptic on 
\[
	\{ r \} \times S^2_{\phi, \theta},
\]
for all $r \in \left(r_e - \de, r_e\right)$, we conclude that all inextendible bicharacteristics of the induced operator pass through all hypersurfaces 
\[
	\{ r \} \times S^2_{\phi, \theta}
\]
precisely once.
Moreover, the induced equation for $z$ is a linear wave operator with real analytic coefficients on the
\[
	\left(r_e - \de, r_e\right)_r \times S^2_{\phi_*, \theta}.
\]
and we know by \textbf{Step 1} that $z$ is real analytic in an open subset
\[
	(r_e - \e, r_e) \times S^2_{\phi_*, \theta}
\]
for some $\e > 0$.
Propagation of analytic singularities, see for example \cite{M2002}*{Thm.\ 4.3.7 \& Exe.\ 4.6.4}, therefore implies that $z$ is real analytic in
\[
	\left(r_e - \de, r_e\right)_r \times S^2_{\phi, \theta},
\]
and hence $u$ is real analytic in 
\[
	\R_t \times \left(r_e - \de, r_e\right)_r \times S^2_{\phi, \theta}.
\]
One similarly treats the subset 
\[
	\R_t \times \left(r_c, r_c + \de\right)_r \times S^2_{\phi, \theta},
\]
in case $\Lambda > 0$.
This finishes the proof.
\end{proof}

\section{Standard quasinormal modes}\label{sec:improved}

We finish by proving our last main result:
\begin{proof}[Proof of Theorem \ref{mainthm: K and K-dS improved}]

As in the proof of Theorem \ref{mainthm: K and K-dS}, let us for simplicity restrict to the case of complex functions, as opposed to complex tensors of higher rank. 
This will again make the proof more transparent and avoid technical details that are completely analogous to the corresponding part of the proof of Theorem \ref{mainthm: general}.
It is convenient to change coordinate system to one that is better suited for the quasinormal mode condition (ii) in Theorem \ref{mainthm: K and K-dS improved}.
We introduce the new coordinate system $(\tau_*, r, \psi_*, \theta)$, where
\begin{equation} \label{eq: star coordinates}
	\begin{pmatrix}
	    \tau_* \\
		\psi_*
	\end{pmatrix} 
		:= 
	\begin{pmatrix}
	    t_* \\
		\phi_* - \frac a{r_0^2 + a^2} t_*
	\end{pmatrix},
\end{equation}
with again $r_0 \in (r_e, r_c)$ is uniquely defined by
\[
    \mu'(r_0) = 0,
\]
and $r_0 = \infty$ if $\Lambda = 0$.
Note that
\[
    \d_{\tau_*}
        = \d_{t_*} + \frac a {r_0^2 + a^2} \d_{\phi_*}, \quad
    \d_{\psi_*}
        = \d_{\phi_*}
\]
are both Killing vector fields, since $a$ and $r_0$ are constant.
It follows that $u$ is a quasinormal mode if and only if 
\[
	u(\tau_*, r, \psi_*, \theta) 
		= e^{-i \s \tau_*}z(r, \psi_*, \theta).
\]
where $z$ is smooth in
\[
	(r_e - \de, r_c + \de)_r \times S^2_{\psi_*, \theta},
\]
where $r_c = \infty$ if $\Lambda = 0$.
Since the coefficients of $P$ are independent of $t_*$ and $\phi_*$, and therefore of $\tau_*$, it follows that $z$ satisfies a $\tau_*$-reduced equation
\begin{equation} \label{eq: t star reduced}
	P_\s z = 0,
\end{equation}
in
\[
	(r_e - \de, r_c + \de)_r \times S^2_{\psi_*, \theta},
\]
where 
\[
    P_\s u 
    		:= e^{i \s \tau_*} P\left(e^{- i \s \tau_*} u \right).
\]
We now further decompose into angular modes
\begin{equation} \label{eq: Fourier decomposition}
	z(r, \psi_*, \theta) = \sum_{k \in \Z} e^{- ik\psi_*}v_k(r, \theta),
\end{equation}
where we claim that each summand
\[
	e^{- ik\psi_*} v_k(r, \theta) = \frac1{2\pi} e^{- ik\psi_*} \int_0^{2\pi} e^{iks}z(r, s, \theta) ds
\]
is smooth on $(r_e - \de, r_c + \de)_r \times S^2_{\psi_*, \theta}$.
Indeed, let $u$ be the unique solution at a point
\[
	(s, r, p) \in \R_s \times (r_e - \de, r_c + \de)_r \times S^2_{\psi_*, \theta}
\]
to
\begin{align*}
	\left( \d_s + ik \right) u(s, r, p)
		&= \frac1{2\pi} z\left(r, \exp_p \left(s\d_{\psi_*} \right) \right), \\
	u(0, r, p)
		&= 0.
\end{align*}
where $\exp(s\d_{\psi_*})$ denotes the flow along $\d_{\psi_*}$ at time $s$, starting at $p$.
Then $u$ is smooth and since
\[
	e^{- ik\psi_*} v_k(r, \theta)
		= u(2\pi, r, p),
\]
where $p = (\psi_*, \theta)$, it is smooth as claimed.
Since the coefficients of $P$ are independent of $t_*$ and $\phi_*$, and therefore of $\psi_*$, and $P_\s$ acts diagonally on the $\psi_*$-Fourier modes, it follows that
\[
	P_\s \left( e^{- ik\psi_*} v_k \right) 
		= 0,
\]
in
\[
	(r_e - \de, r_c + \de)_r \times S^2_{\psi_*, \theta},
\]
for each $k \in \Z$.

Now, if $\Lambda > 0$, then \cite{PV2021}*{Thm.\ 2.1} implies that the operator $P_\s$ is a Fredholm operator between appropriate function spaces containing $e^{- ik\psi_*} v_k$.
Since the kernel is finite dimensional, it follows that only finitely many such terms can be non-zero.
We conclude that
\begin{equation} \label{eq: Fourier decomposition finite}
	z(r, \psi_*, \theta) 
		= \sum_{j =1}^N e^{- ik_j\psi_*}v_{k_j}(r, \theta)
\end{equation}
and therefore
\begin{align*}
	u(\tau_*, r, \psi_*, \theta)
		&= \sum_{j =1}^N e^{- i \left( \s \tau_* + k_j \psi_* \right)}v_{k_j}(r, \theta) \\
		&= \sum_{j =1}^N e^{- i \left( \left(\s - \frac a{r_0^2 + a^2}k_j \right) t_* + k_j \phi_* \right)}v_{k_j}(r, \theta).
\end{align*}
Each term satisfies the assumption of Theorem \ref{mainthm: K and K-dS} and are therefore analytic. 
Hence the \emph{finite} sum is also real analytic, concluding the proof when $\Lambda > 0$.

In order to similarly proceed in the case $\Lambda = 0$, we need to instead use the Fredholm theory developed in \cite{V2018} (remember that the above coordinate change is trivial when $\Lambda = 0$) to deduce that in fact 
\begin{equation}
	z(r, \phi_*, \theta) 
		= \sum_{j =1}^N e^{- ik_j\phi_*}v_{k_j}(r, \theta)
\end{equation}
is a \emph{finite} sum.
In this case we have $r_c = \infty$ and the cosmological horizon is replaced by an asymptotically Euclidean end.
For the analysis near the event horizon, the methods based on \cite{V2013} described above can be applied without changes. 
However, the analysis near the asymptotically Euclidean end cannot be based on \cite{V2013}, we instead need to use a slight generalization of \cite{V2018}*{Prop.\ 5.28}.
Let us therefore briefly recall how the Fredholm problem was set up in \cite{V2018}*{Prop.\ 5.28}.
We begin by bordifying the space
\[
	(r_e - \de, \infty)_r \times S^2_{\phi_*, \theta}, 
\] 
at $r = \infty$ by introducing $x := \frac1r$, i.e.\ we radially compactify spacelike infinity.
We thus write
\[
	\V := \left[0, \frac1{r_e - \de}\right)_x \times S^2_{\phi_*, \theta} \subset \overline{\R^3}, 
\]
where $\overline {\R^3}$ is the radially compactified $\R^3$.
On these spaces, we define
\[
	\mathcal Y_{sc}^{s, l} 
		:= \left\{ u|_{\V} \mid u \in H^{s, l}(\overline{\R^3}) \right\},
\]
where $s, l$ are variable order differential and decay orders (as $x \to 0$), which we will choose below.
We refer to \cite{V2018}*{Sec.\ 5.3.9} for the definition of variable order weighted Sobolev spaces $H^{s, l}(\overline{\R^3})$.
Note that near the spacelike hypersurface 
\[
	\left\{x = \frac1{r_e - \de}\right\}, 
\]
$\mathcal Y_{sc}^{s, l}$ is similar to $\mathcal Y^s$ introduced above.
Analogous to above, define
\[
	\mathcal X_{sc}^{s, l} 
		:= \left\{u \in \mathcal Y_{sc}^{s,l} \mid \hat P u \in \mathcal Y_{sc}^{s-1,l+1} \right\}
\]
and consider
\begin{equation} \label{eq: Fredholm Lambda 0}
	\hat P: \mathcal X_{sc}^{s, l} \to \mathcal Y_{sc}^{s-1,l+1}.
\end{equation}
The characteristic set of $\hat P$ has two components, one close to the event horizon and a scattering characteristic set at $x = 0$, in particular, the characteristic set at fiber infinity near $x = 0$ is empty.
By the decay assumptions on $A$ and $B$, the scattering principal symbol of $\hat P$ at $x = 0$ is given by
\[
	\p|_{x = 0}(\xi) = \abs{\xi}^2 \Id - \sigma^2,
\]
for the \emph{fixed} $\sigma$ with $\Im \sigma \geq 0$, and any
\[
	\xi \in {}^{sc}T^*_{\{x = 0\}}\V.
\]
If $\Im \sigma > 0$ it follows that $\sigma^2 \notin [0, \infty)$, which implies that $\hat P$ is \emph{elliptic} as set up in \eqref{eq: Fredholm Lambda 0} and consequently a Fredholm operator for any order $s, l$.
However, in case $\sigma \in \R \backslash \{0\}$, there is a scattering characteristic set at $x = 0$, given by all $\xi \in {}^{sc}T^*_{\{x = 0\}}\V$ with $\abs{\xi} = \abs \sigma$.
As shown in \cite{V2018}*{p.\ 311--314}, the sets 
\[
	L_\pm = \left\{(y, \xi) \in {}^{sc}T^*_{\{x = 0\}}\V \mid y = c \xi, \abs{\xi}^2 = \sigma^2, \pm c > 0\right\},
\]
act as a source and a sink, respectively, for the Hamiltonian flow.
It is also shown that \eqref{eq: Fredholm Lambda 0} is a Fredholm operator (c.f.\ \cite{V2018}*{Prop.\ 5.28}), if $l$ is chosen such that either
\[
	l|_{L_+} < - \frac12 \text{ and } l|_{L_-} > -\frac12
\]
or the other way around (with $L_+$ and $L_-$ swapped).
The decay assumption on $u|_{t_* = 0}$ in Theorem \ref{mainthm: K and K-dS improved} ensures that $z|_{t_* = 0} \in \mathcal X^{s,l}_{sc}$, and therefore each summand in \eqref{eq: Fourier decomposition}, is in $\ker(\hat P)$ as set up in \eqref{eq: Fredholm Lambda 0}.
We have thus proven \eqref{eq: Fourier decomposition finite}, for the case when $\Lambda = 0$ and $\sigma \neq 0$ (and $\Im \sigma \geq 0$).

The case which remains is when $\Lambda = \sigma = 0$.
The structure of the operator $\hat P$ now changes drastically near $x = 0$ and is more naturally thought of as a b-operator in the sense of Melrose \cite{Me1993}, see also \cites{GH2008, GH2009}.
We follow \cite{V2018}*{Sec.\ 5.6} for the Fredholm theory.
Concretely, we note that the fast decay assumptions on $A$ and $B$ ensure that
\[
	x^{-\frac{n-2}2}x^{-2} \hat P x^{\frac{n-2}2}
\]
is a b-operator with normal operator 
\[
	- (x\d_x)^2 + \Delta_h + \frac{(n-2)^2}4
\]
at $x = 0$.
Choose a smooth function $f: [0, \infty) \to [0, \infty)$, such that $f(x) = x$ for $x \leq \e$ and $f(x) = 1$ for $x \geq 2\e$ and define
\[
	L := f(x)^{-\frac{n-2}2}f(x)^{-2} \hat P f(x)^{\frac{n-2}2},
\]
where $\e > 0$ small enough so that the component of the characteristic set of $\hat P$ away from $x = 0$ is unaffected by this conjugation.
We now define the spaces
\[
	\mathcal Y_b^{s, l} := \{u|_{\V} \mid u \in H^{s, l}_b(\overline{\R^3})\},
\]
where $H^{s, l}_b(\overline{\R^3})$ is defined in \cite{V2018}*{p.\ 353} and
\[
	\mathcal X_b^{s, l} := \left\{u \in \mathcal Y_b^{s,l} \mid L u \in \mathcal Y_b^{s-1,l} \right\}.
\]
By combining the discussion on \cite{V2018}*{p.\ 361} (c.f.\ also \cite{V2018}*{Thm.\ 5.11}) with the theory near the event horizon described above, we know that
\[
	L: \mathcal X_b^{s, l} \to \mathcal Y_b^{s-1, l} 
\]
is a Fredholm for all $s, l \in \R$, such that
\[
	l^2 - \frac{(n-2)^2}4
\]
is not an $L^2$ eigenvalue of $\Delta$ on the $2$-sphere. 
Since the set of $L^2$-eigenvalues is discrete, we can choose $l$ arbitrarily large and still have a Fredholm operator.
It follows that the kernel of $L$ is finite dimensional.
Now, the kernel of $\hat P$ and the kernel of $L$ are related just by a multiplication with $f(x)^{\frac{n-2}2}$ and we have thus proven the $\ker(\hat P)$ is finite dimensional and consequently \eqref{eq: Fourier decomposition finite}.
This finishes the proof.
\end{proof}

\appendix

\section{Surface gravity of a Killing horizon}
Let us verify the claim in Remark \ref{rmk: nowhere vanishing surface gravity} about the surface gravity of a Killing horizon:

\begin{lemma} \label{le: surface gravity}
Consider a smooth spacetime $(M,g)$, with a smooth connected lightlike hypersurface $\H \subset M$ and a smooth Killing vector field $W$ on $M$, such that $W|_\H$ is nowhere vanishing, lightlike and tangent to $\H$.
If \eqref{eq: constant surface gravity condition} is satisfied, then there is a constant $\kappa \in \R$, such that
\[
	\n_W W|_\H = \kappa W|_\H.
\]
\end{lemma}
\begin{proof}
Using that $W|_\H$ is lightlike and tangent to $\H$, we compute that for all vector fields $X, Y$, tangent to $\H$, we have
\begin{align*}
	g(\n_XW, Y)|_\H 
		&= \frac12\L_Wg(X, Y)|_\H  + \frac12 \left(g(\n_XW, Y)|_\H - g(\n_YW, X)|_\H\right) \\
		&= \frac12 \left(X g(W, Y)|_\H  - Yg(W, X)|_\H  - g(W, [X, Y])|_\H\right) \\
		&= 0,
\end{align*}
since also $[X, Y]$ is tangent to $\H$.
Hence $\n_X W$ is tangent to $\H$ and normal to $\H$, meaning that there is a unique one-form $\o$ on $\H$, such that
\[
	\n_X W|_\H = \o(X) W|_\H.
\]
The assertion in the lemma is thus that $\omega(W|_\H)$ is constant.
Since $W$ is a Killing vector field, with $W|_\H$ tangent to $\H$, it is immediate that
\[
	\L_W\o|_\H = 0.
\]
For any $X \in T\H$, we have
\begin{align*}
	X(\omega(W|_\H))
		&= \md \omega(X, W|_\H) + W|_\H \left( \omega(X) \right) + \omega([X, W|_\H]) \\
		&= \md \omega(X, W|_\H) + \L_{W|_\H} \omega(X) \\
		&= \md \omega(X, W|_\H).
\end{align*}
It thus remains to show that $\md \omega(X, W|_\H) = 0$, for all $X \in T\H$.
For this, we first note that for all $X, Y \in T\H$, we have
\begin{align*}
	R(X, Y)W|_\H
		&= \n_X \n_Y W|_\H - \n_Y \n_X W|_\H - \n_{[X, Y]}W|_\H \\
		&= \n_X(\omega(Y)W|_\H) - \n_Y(\omega(X)W|_\H) - \omega([X, Y])W|_\H \\
		&= X (\omega(Y))W|_\H + \omega(Y)\omega(X)W|_\H - Y(\omega(X))W|_\H \\
		&\qquad - \omega(X)\omega(Y)W|_\H - \omega([X, Y])W|_\H \\
		&= \md\omega(X,Y) W|_\H.
\end{align*}
Let $e_0 := W|_\H, e_2, \hdots, e_n$ locally span $T\H$ and let $e_1$ be the unique locally defined vector field along $\H$, transversal to $\H$, such that
\[
	g(e_1, e_0)|_\H = 1, \quad g(e_1, e_j)|_\H = 0,
\]
for $j = 1, \hdots, n$. 
We now trace the curvature expression using this local frame, with any $X \in T\H$, and compute
\begin{align*}
	\Ric(X, W)|_\H
		&= \sum_{\a, \b = 0}^n g^{\a \b} R\left(e_\a, X, W, e_\b\right)|_\H \\
		&= R\left(W, X, W, e_1\right)|_\H + R\left(e_1, X, W, W\right)|_\H \\*
		&\qquad + \sum_{i, j = 2}^n g^{ij} R\left(e_i, X, W, e_j\right)|_\H \\
		&= \md \o(W|_\H, X) g(W, e_1) + \sum_{i, j = 2}^n g^{ij} \md \o(e_i, X) g(W, e_j)|_\H \\
		&= \md \o(W|_\H, X).
\end{align*}
We therefore conclude that 
\begin{align*}
	X \kappa
		&= X(\o(W|_\H)) \\
		&= \md \o(W|_\H, X) \\
		&= \Ric(X, W)|_\H \\
		&= 0,
\end{align*}
for all $X \in T\H$, which proves that $\kappa$ is constant.
\end{proof}

\end{sloppypar}

\begin{bibdiv}
\begin{biblist}

\bib{AIK2010}{article}{
   author={Alexakis, Spyros},
   author={Ionescu, Alexandru D.},
   author={Klainerman, Sergiu},
   title={Hawking's local rigidity theorem without analyticity},
   journal={Geom. Funct. Anal.},
   volume={20},
   date={2010},
   number={4},
   pages={845--869},
}

\bib{AIK2010_2}{article}{
   author={Alexakis, Spyros},
   author={Ionescu, Alexandru D.},
   author={Klainerman, Sergiu},
   title={Uniqueness of smooth stationary black holes in vacuum: small
   perturbations of the Kerr spaces},
   journal={Comm. Math. Phys.},
   volume={299},
   date={2010},
   number={1},
   pages={89--127},
}

\bib{AIK2014}{article}{
   author={Alexakis, Spyros},
   author={Ionescu, Alexandru D.},
   author={Klainerman, Sergiu},
   title={Rigidity of stationary black holes with small angular momentum on
   the horizon},
   journal={Duke Math. J.},
   volume={163},
   date={2014},
   number={14},
   pages={2603--2615},
}



\bib{BH2008}{article}{
  AUTHOR = {Bony, Jean-Fran{\c{c}}ois},
  author={H{\"a}fner, Dietrich},
     TITLE = {Decay and non-decay of the local energy for the wave equation
              on the de {S}itter-{S}chwarzschild metric},
   JOURNAL = {Comm. Math. Phys.},
    VOLUME = {282},
      YEAR = {2008},
    NUMBER = {3},
     PAGES = {697--719},
}

\bib{BR2021}{article}{
   author={Bustamante, I.},
   author={Reiris, M.},
   title={On the existence of Killing fields in smooth spacetimes with a compact Cauchy horizon},
   journal={Class. Quant. Grav.},
   volume={38},
   date={2021},
   number={7},
   pages={Paper No. 075010, 16},
}


\bib{D2011}{article}{
    AUTHOR = {Dyatlov, Semyon},
     TITLE = {Quasi-normal modes and exponential energy decay for the
              {K}err-de {S}itter black hole},
   JOURNAL = {Comm. Math. Phys.},
    VOLUME = {306},
      YEAR = {2011},
    NUMBER = {1},
     PAGES = {119--163},
}


\bib{D2012}{article}{
    AUTHOR = {Dyatlov, Semyon},
     TITLE = {Asymptotic distribution of quasi-normal modes for {K}err-de {S}itter black holes},
	JOURNAL ={Ann. Henri Poincare},
	volume = {13}, 
	pages={1101--1166},
	year={2012}
}

\bib{FRW1999}{article}{
   author={Friedrich, Helmut},
   author={R\'acz, Istv\'an},
   author={Wald, Robert~M.},
   title={On the rigidity theorem for spacetimes with a stationary event
   horizon or a compact Cauchy horizon},
   journal={Comm. Math. Phys.},
   volume={204},
   date={1999},
   number={3},
   pages={691--707},
}

\bib{GW2020}{article}{
  AUTHOR = {Gajic, Dejan},
  author={Warnick, Claude},
     TITLE = {A model problem for quasinormal ringdown of asymptotically flat or extremal black holes},
   JOURNAL = {J. Math. Phys.},
    VOLUME = {61},
      YEAR = {2020},
    NUMBER = {10},
}

\bib{GZ2020}{article}{
   author={Galkowski, Jeffrey},
   author={Zworski, Maciej},
   title={Analytic hypoellipticity of Keldysh operators},
   journal={Proc. Lond. Math. Soc. (3)},
   volume={123},
   date={2021},
   number={5},
   pages={498--516},
}
	

\bib{GH2008}{article}{
   author={Guillarmou, Colin},
   author={Hassell, Andrew},
   title={Resolvent at low energy and Riesz transform for Schr\"{o}dinger
   operators on asymptotically conic manifolds. I},
   journal={Math. Ann.},
   volume={341},
   date={2008},
   number={4},
   pages={859--896},
}

\bib{GH2009}{article}{
   author={Guillarmou, Colin},
   author={Hassell, Andrew},
   title={Resolvent at low energy and Riesz transform for Schr\"{o}dinger operators on asymptotically conic manifolds. II},
   journal={Ann. Inst. Fourier (Grenoble)},
   volume={59},
   date={2009},
   number={4},
   pages={1553--1610},
}

\bib{GM2021}{article}{
   author={Gurriaran, Sebastian},
   author={Minguzzi, Ettore},
   title={Surface gravity of compact non-degenerate horizons under the dominant energy condition},
      journal={Comm. Math. Phys.},
   volume={395},
   date={2022},
   number={2},
   pages={679--713},
}

\bib{H2014}{article}{
   author={Haber, Nick},
   title={A normal form around a Lagrangian submanifold of radial points},
   journal={Int. Math. Res. Not. IMRN},
   date={2014},
   number={17},
   pages={4804--4821},
}

\bib{H1972}{article}{
   author={Hawking, Stephen~W.},
   title={Black holes in general relativity},
   journal={Comm. Math. Phys.},
   volume={25},
   date={1972},
   pages={152--166},
}

\bib{HE1973}{book}{
   author={Hawking, Stephen W.},
   author={Ellis, George~F.\ R.},
   title={The large scale structure of space-time},
   note={Cambridge Monographs on Mathematical Physics, No. 1},
   publisher={Cambridge University Press, London},
   date={1973},
   pages={xi+391},
}

\bib{HV2015}{article}{
  AUTHOR = {Hintz, Peter},
  author={Vasy, Andr{\'a}s},
     TITLE = {Semilinear wave equations on asymptotically de {S}itter,
              {K}err--de {S}itter and {M}inkowski spacetimes},
   JOURNAL = {Anal. PDE},
    VOLUME = {8},
      YEAR = {2015},
    NUMBER = {8},
     PAGES = {1807--1890},
   }

\bib{HV2018}{article}{
  AUTHOR = {Hintz, Peter},
  author={Vasy, Andr\'as},
     TITLE = {The global non-linear stability of the {K}err--de {S}itter
              family of black holes},
   JOURNAL = {Acta Math.},
    VOLUME = {220},
      YEAR = {2018},
    NUMBER = {1},
     PAGES = {1--206},
}

\bib{HIW2007}{article}{
   author={Hollands, Stefan},
   author={Ishibashi, Akihiro},
   author={Wald, Robert~M.},
   title={A higher dimensional stationary rotating black hole must be
   axisymmetric},
   journal={Comm. Math. Phys.},
   volume={271},
   date={2007},
   number={3},
   pages={699--722},
}

\bib{IonescuKlainerman2013}{article}{
   author={Ionescu, Alexandru~D.},
   author={Klainerman, Sergiu},
   title={On the local extension of Killing vector-fields in Ricci flat
   manifolds},
   journal={J. Amer. Math. Soc.},
   volume={26},
   date={2013},
   number={2},
   pages={563--593},
}

\bib{IM1985}{article}{
   author={Isenberg, James},
   author={Moncrief, Vincent},
   title={Symmetries of cosmological Cauchy horizons with exceptional
   orbits},
   journal={J. Math. Phys.},
   volume={26},
   date={1985},
   number={5},
   pages={1024--1027},
}


\bib{KIS2000}{article}{
  author = {Kodama, Hideo},
  author={Ishibashi, Akihiro},
  author={Seto, Osamu},
  title = {Brane world cosmology: gauge-invariant formalism for perturbation},
  journal = {Phys. Rev. D (3)},
   volume = {62},
   year = {2000},
   number = {6},
}

\bib{LZ2019}{article}{
   author={Lebeau, Gilles},
   author={Zworski, Maciej},
   title={Remarks on Vasy's operator with analytic coefficients},
   journal={Proc. Amer. Math. Soc.},
   volume={147},
   date={2019},
   number={1},
   pages={145--152},
}

\bib{M2002}{book}{
   author={Martinez, Andr\'{e}},
   title={An introduction to semiclassical and microlocal analysis},
   series={Universitext},
   publisher={Springer-Verlag, New York},
   date={2002},
   pages={viii+190},
}


\bib{Me1993}{book}{
   author={Melrose, Richard B.},
   title={The Atiyah-Patodi-Singer index theorem},
   series={Research Notes in Mathematics},
   volume={4},
   publisher={A K Peters, Ltd., Wellesley, MA},
   date={1993},
   pages={xiv+377},
}

\bib{M1994}{article}{
   author={Melrose, Richard B.},
   title={Spectral and scattering theory for the Laplacian on asymptotically Euclidian spaces},
   conference={
      title={Spectral and scattering theory},
      address={Sanda},
      date={1992},
   },
   book={
      series={Lecture Notes in Pure and Appl. Math.},
      volume={161},
      publisher={Dekker, New York},
   },
   date={1994},
   pages={85--130},
}

\bib{MI1983}{article}{
   author={Moncrief, Vincent},
   author={Isenberg, James},
   title={Symmetries of cosmological Cauchy horizons},
   journal={Comm. Math. Phys.},
   volume={89},
   date={1983},
   number={3},
   pages={387--413},
}

\bib{MI2008}{article}{
   author={Moncrief, Vincent},
   author={Isenberg, James},
   title={Symmetries of higher dimensional black holes},
   journal={Classical Quantum Gravity},
   volume={25},
   date={2008},
   number={19},
   pages={195015, 37},
}

\bib{MI2018}{article}{
   author={Moncrief, Vincent},
   author={Isenberg, James},
   title={Symmetries of cosmological Cauchy horizons with non-closed orbits},
   journal={Comm. Math. Phys.},
   volume={374},
   date={2020},
   number={1},
   pages={145--186},
}


\bib{P2018}{article}{
   author={Petersen, Oliver~L.},
   title={Wave equations with initial data on compact Cauchy horizons},
   journal={Anal. PDE},
   volume={14},
   date={2021},
   number={8},
   pages={2363--2408},
}


\bib{P2019}{article}{
   author={Petersen, Oliver~L.},
   title={Extension of Killing vector fields beyond compact Cauchy horizons},
   journal={Adv. Math.},
   volume={391},
   date={2021},
   pages={Paper No. 107953, 65},
}

\bib{PR2018}{article}{
   author={Petersen, Oliver},
   author={R\'{a}cz, Istv\'{a}n},
   title={Symmetries of vacuum spacetimes with a compact Cauchy horizon of constant non-zero surface gravity},
   JOURNAL = {Ann.~Henri Poincar\'e},
    VOLUME = {24},
      YEAR = {2023},
     PAGES = {3921--3943},
}


\bib{PV2021}{article}{
   author={Petersen, Oliver},
   author={Vasy, Andr\'as},
   title={Wave equations in the Kerr-de Sitter spacetime: the full subextremal range},
   journal={J.~Eur.~Math.~Soc. (to appear) arXiv: 2112.01355},
}

\bib{R2000}{article}{
   author={R\'{a}cz, Istv\'{a}n},
   title={On further generalization of the rigidity theorem for spacetimes
   with a stationary event horizon or a compact Cauchy horizon},
   journal={Classical Quantum Gravity},
   volume={17},
   date={2000},
   number={1},
   pages={153--178},
}


\bib{RW1957}{article}{
  author={Regge, Tullio},
  author={Wheeler, John~A.},
title={{S}tability of a {S}chwarzschild {S}ingularity},
journal={Phys. Rev.},
volume={108},
pages={1063--1069},
year={1957},
}

\bib{SBZ1997}{article}{
  AUTHOR = {S{\'a} Barreto, Ant{\^o}nio},
  author={Zworski, Maciej},
TITLE = {Distribution of resonances for spherical black holes},
JOURNAL = {Math. Res. Lett.},
VOLUME = {4},
YEAR = {1997},
NUMBER = {1},
PAGES = {103--121},
}

	
\bib{SR2015}{article}{
    AUTHOR = {Shlapentokh-Rothman, Yakov},
     TITLE = {Quantitative mode stability for the wave equation on the
              {K}err spacetime},
   JOURNAL = {Ann. Henri Poincar\'e},
    VOLUME = {16},
      YEAR = {2015},
    NUMBER = {1},
     PAGES = {289--345},
}

\bib{V2013}{article}{
   author={Vasy, Andr\'{a}s},
   title={Microlocal analysis of asymptotically hyperbolic and Kerr-de
   Sitter spaces (with an appendix by Semyon Dyatlov)},
   journal={Invent. Math.},
   volume={194},
   date={2013},
   number={2},
   pages={381--513},
}


\bib{V2018}{article}{
	author={Vasy, Andr\'{a}s}, 
	title={A Minicourse on Microlocal Analysis for Wave Propagation},
	book={
		title={Asymptotic Analysis in General Relativity},
		series={London Mathematical Society Lecture Note Series}, 
		publisher={Cambridge University Press}, 
		place={Cambridge}, 
		year={2018}, 
		editor={Daudé, Thierry and Häfner, Dietrich and Nicolas, Jean-PhilippeEditors},
	},
	pages={219--374},
}



\bib{V2019}{article}{
   author={Vasy, Andr\'{a}s},
   title={Limiting absorption principle on Riemannian scattering
   (asymptotically conic) spaces, a Lagrangian approach},
   journal={Comm. Partial Differential Equations},
   volume={46},
   date={2021},
   number={5},
   pages={780--822},
}

\bib{VZ2000}{article}{
   author={Vasy, Andr\'{a}s},
   author={Zworski, Maciej},
   title={Semiclassical estimates in asymptotically Euclidean scattering},
   journal={Comm. Math. Phys.},
   volume={212},
   date={2000},
   number={1},
   pages={205--217},
}

 
\bib{V1970}{article}{
  author={Vishveshwara, C.~V.},
title={{S}tability of the {S}chwarzschild {M}etric},
journal={Phys. Rev. D},
volume={1},
pages={2870--2879},
year={1970}
}


\bib{W1989}{article}{
    AUTHOR = {Whiting, Bernard F.},
     TITLE = {Mode stability of the {K}err black hole},
   JOURNAL = {J. Math. Phys.},
    VOLUME = {30},
      YEAR = {1989},
    NUMBER = {6},
     PAGES = {1301--1305},
}
 
 \bib{Z1970}{article}{
author={Zerilli, Frank~J.},
title={{E}ffective {P}otential for {E}ven-{P}arity {R}egge--{W}heeler
  {G}ravitational {P}erturbation {E}quations},
journal={Phys. Rev. Lett.},
volume={24},
pages={737--738},
year={1970}
}


\bib{Zu2017}{article}{
   author={Zuily, Claude},
   title={Real analyticity of radiation patterns on asymptotically
   hyperbolic manifolds},
   journal={Appl. Math. Res. Express.},
   date={2017},
   number={2},
   pages={386--401},
}


\end{biblist}
\end{bibdiv}
\end{document}